\documentclass[12pt,reqno]{amsart}
\usepackage{amsmath, mathtools}
\usepackage{amsthm}
\usepackage{amssymb}
\usepackage{enumerate}
\usepackage{amscd}
\usepackage{color}
\usepackage{pb-diagram}
\usepackage{graphicx}
\usepackage[all, cmtip]{xy}
\usepackage{soul}
\usepackage[colorinlistoftodos]{todonotes}
\usepackage{esint}
\usepackage{hyperref}
\hypersetup{pdftex,colorlinks=true,allcolors=blue}
\usepackage{hypcap}
\usepackage[titletoc]{appendix}
\usepackage{comment}
\usepackage{float}
\restylefloat{table}
\usepackage{tikz}
\usepackage{tikz-cd}
\usetikzlibrary{calc,decorations.pathmorphing,shapes}

\usepackage[colorinlistoftodos]{todonotes}
\usepackage[normalem]{ulem}
\usepackage{soul,xcolor}
\setstcolor{red}
\makeatletter
\def\uwave{\bgroup \markoverwith{\lower3.5\p@\hbox{\sixly \textcolor{red}{\char58}}}\ULon}
\font\sixly=lasy6 
\makeatother

\newenvironment{red}{\relax\color{red}}{\relax}
\newenvironment{blue}{\relax\color{blue}}{\hspace*{.5ex}\relax}

\newcommand{\ber}{\begin{red}}
	\newcommand{\er}{\end{red}}
\newcommand{\beb}{\begin{blue}}
	\newcommand{\eb}{\end{blue}}
\newcounter{sarrow}

\newtheorem{theo}{Theorem}[section]
\newtheorem{lemma}{Lemma}[section]
\newtheorem{prop}[lemma]{Proposition}
\newtheorem{coro}[lemma]{Corollary}
\newtheorem{lemm}[lemma]{Lemma}
\theoremstyle{definition}
\newtheorem{rema}[lemma]{Remark}
\newtheorem{defi}[lemma]{Definition}
\newtheorem{thm-Intro}{Theorem} 
\newtheorem{cor-Intro}{Corollary} 

\numberwithin{equation}{section}

\newcommand{\abs}[1]{\left|#1\right|} 

\newcommand{\Hess}{\operatorname{Hess}}

\newcommand{\twobar}{/\kern-0.5em/}
\newcommand{\threebar}{/\kern-0.5em/\kern-0.5em/}

\title[Stochastic Differential geometry and fundamental gap]{Probabilistic Method to Fundamental gap problems 
on the sphere}

\author{Gunhee Cho}
\address{Department of Mathematics\\
	University of California, Santa Barbara\\
	Santa Barbara, CA 93106.}
\email{gunhee.cho@math.ucsb.edu}
\thanks{GC was partially supported by AMS Simons travel grant}
\author{Guofang Wei}
\address{Department of Mathematics\\
	University of California, Santa Barbara\\
	Santa Barbara, CA 93106.}
\email{wei@math.ucsb.edu}
\thanks{ GW was partially supported by NSF DMS grant 2104704.}
\author{Guang Yang}
\address{Department of Mathematics\\
	Purdue University, West Lafayette, IN 47907.}
\email{yang2220@purdue.edu}

\begin{document}

	\begin{abstract} 
		
We provide a probabilistic proof of the fundamental gap estimate for Schr\"odinger operators in convex domains on the sphere, which extends the probabilistic proof of F. Gong, H. Li, and D. Luo for the Euclidean case. Our results further generalize the results achieved for the Laplacian by S. Seto, L. Wang, and G. Wei, as well as by C. He, G. Wei, and Qi S. Zhang. The essential ingredient in our analysis is the reflection coupling method on Riemannian manifolds.	
	\end{abstract}
	
	\maketitle
	
	\section{Introduction}
	
	In the context of a bounded smooth domain $\Omega$ within a Riemannian manifold, the eigenvalues of the Laplacian or more generally, Schr\"odinger operators on $\Omega$ subject to Dirichlet boundary conditions are ordered as follows:
	$$
	\lambda_{1}<\lambda_{2} \leq \lambda_{3} \cdots \rightarrow \infty.
	$$
	
	The estimate of the gap between the first two eigenvalues, known as the fundamental (or mass) gap, denoted as
	$$
\Gamma(\Omega)= \lambda_{2}-\lambda_{1}>0,
$$
holds significant importance in both mathematics and physics, and has been a subject of  active  study.

In a noteworthy achievement, using two point maximum principle B. Andrews and J. Clutterbuck successfully proved the fundamental gap conjecture for convex domains in Euclidean space \cite{MR2784332}. Namely, for Schr\"odinger operators with convex potential, we have $\Gamma(\Omega) \ge \tfrac{3\pi^2}{D^2}$ for convex domains $\Omega \subset \mathbb R^n$, where $D$ is the diameter of the domain.  More recently, the second author, along with coauthors, extended the same estimate to convex domains on the sphere for the Laplacian \cite{MR3960269, MR4124117, MR4349140}.

An essential step in proving the fundamental gap conjecture is the log-concavity estimate of the first eigenfunction. The following definition was firstly introduced by B. Andrews and J. Clutterbuck \cite{MR2784332}.  
An even function $\tilde{V} \in C^{1}([-D/2, D / 2])$ is called a modulus of concavity for $V \in C^{1}(\Omega)$, if for all $x, y \in \Omega, x \neq y$, one has
\begin{equation}\label{eq:mod-of-concavity}
\left\langle\nabla V(y), \gamma'(y) \right\rangle-\left\langle\nabla V(x), \gamma'(x) \right\rangle \leq 2 \tilde{V}^{\prime}\left(\frac{\rho(x,y)}{2}\right),
\end{equation}
where $\gamma$ represents the minimizing geodesic that connects $x, y$ and $\rho$ denotes the geodesic distance function. Intuitively, this says that $V$ is ``more concave" than $\tilde{V}$. If the inequality \eqref{eq:mod-of-concavity} is reversed
\begin{equation}\label{eq:mod-of-convexity}
   \left\langle\nabla V(y), \gamma'(y) \right\rangle-\left\langle\nabla V(x), \gamma'(x) \right\rangle \geq 2 \tilde{V}^{\prime}\left(\frac{\rho(x,y)}{2}\right), 
\end{equation}
we say $\tilde{V} $ is a modulus of convexity for $V$. 

Let $\phi_1$ be the first positive Dirichlet eigenfunction of the Schr\"odinger operator $-\Delta + V$ on $\Omega$ and $\bar{\phi}_1$ the first Dirichlet eigenfunction of the one-dimensional model given in \eqref{one dim model}. The key estimates in obtaining the fundamental gap lower bound is the following log-concavity estimate of the first eigenfunction. 
	\[  \langle \nabla \log \phi_1(y), \gamma'(y)  \rangle -\langle \nabla\log \phi_1(x), \gamma'(x)  \rangle\leq 2\log \bar{\phi}_1'(\rho(x,y)/2)  \]
 for any $x \neq y \in \Omega$. Such log-concavity estimate was proved for Schr\"odinger operators in convex domains $\Omega \subset \mathbb R^n$ in \cite{MR2784332}, and for the Laplacian in convex domains $\Omega \subset \mathbb S^n$ with diameter $D<\pi/2$ in \cite{ MR3960269}. 
 The diameter restriction was removed 
 in \cite{MR4124117}. 
 

	
	In this paper we give a probabilistic proof of this estimate. Denote by $\mathbb M^n_k$ the $n$-dimensional simply connected manifolds with constant sectional curvature $k$. Our first main result is the next 
	\begin{theo} \label{log-concavity} 
		Suppose $\Omega\subset  \mathbb M^n_k$ $(k\ge 0)$ is a bounded strictly convex domain with diameter $D<\pi/\sqrt{k}$. Let $\lambda _1, \phi_1>0$ be the first Dirichlet eigenvalue,  eigenfunction of $-\Delta+V$ on $\Omega$ 
  and $\bar{\lambda}_1, \bar{\phi}_1$ be the first Dirichlet eigenvalue,  eigenfunction of the one-dimensional model
	\begin{equation}
	    - \frac{\partial^2}{\partial s^2}+(n-1)tn_k(s)\frac{\partial}{\partial s}+\Tilde{V}(s),\ s\in [-D/2,D/2],  \label{one dim model}
      \end{equation}
      where $tn_k$ is defined in \eqref{tn_k}. 
      If $\Tilde{V}$ is a modulus of convexity of $V$, and 
      \begin{equation}
         tn_k \left[ (2\lambda_1 -V(x) -V(y) ) -  2 (\bar{\lambda}_1  -\tilde{V} (\rho(x,y)/2)) \right]\ge 0,  \label{eqn:condition om lambda and V}
      \end{equation}then $\log\bar{\phi}_1$ is a modulus of concavity of $\log\phi_1$.   
	\end{theo}

Note that $tn_k =0$ when $k=0$, so (\ref{eqn:condition om lambda and V}) is automatically satisfied for $\mathbb R^n$. When $k=1$ and $V=0$, by choosing $\tilde{V}=0$, (\ref{eqn:condition om lambda and V}) becomes $\lambda_1 \ge \bar{\lambda}_1$, which holds by Lemma 3.12 in \cite{MR3960269}. Therefore the above theorem recovers the logconcavity estimates in  \cite{MR2784332, MR3960269, MR4124117},  
extends the results in \cite{MR3960269, MR4124117} for  the Laplacian operator to Schr\"odinger operators. Condition (\ref{eqn:condition om lambda and V}) indicates interesting difference between $\mathbb R^n$ and $\mathbb S^n$ for Schr\"odinger operators. In fact it is not clear if \cite[Remark 1.4]{MR3960269} holds as stated.

As mentioned before, our argument is based on the reflection coupling method (see for instance, \cite{MR1120916, kendall, Hsu}), which has been used in \cite{MR3489850} to give a probabilistic proof for convex domain in the Euclidean space. Some of our computations are similar to that of \cite{MR3960269, MR4124117}, but the whole proof is much simplified. Namely, the parabolic analysis in \cite{MR4124117} is not needed at all. Also we do not need the initial analysis or log-concavity assumption of the first eigenfunction as in \cite{MR3489850}. It is also worth mentioning that coupling methods have  been used to derive lower bounds for the first nonzero Neumann eigenvalue \cite{ chen1997general, MR4386037, MR1308707}.


With the log-concavity estimate, another ingredient for getting  the fundamental gap estimate is the following gap comparison of \cite[Theorem 4.1]{MR3960269}. We extend it to Schr\"odinger operators with a probabilistic proof, which simplifies the original proof.
	\begin{theo}\label{thm:comparison}
        Let $(M^n,g)$ be a complete Riemannian manifold with Ricci curvature lower bound $(n-1)k, k\in \mathbb{R}$ and $\Omega\subset M^n$ be a convex domain with diameter $D>0$ ($D<\pi/\sqrt{k}$ if $k>0$). Let $\lambda_1,\lambda_2$ be the first two Dirichlet eigenvalues of the Schr\"odinger operator $-\Delta+V$ on $\Omega$ and $\phi_1$ be the corresponding first eigenfunction. Let $\Bar{\lambda}_1, \Bar{\lambda}_2$ be the first two Dirichlet eigenvalues of the one-dimensional model \eqref{one dim model} on $[-D/2,D/2]$ and $\Tilde{\phi}_1$ be the corresponding first eigenfunction. If $\log\Tilde{\phi}_1$ is a modulus of concavity of $\log\phi_1$, then 
        \[ \Gamma(\Omega)=\lambda_{2}-\lambda_{1} \geq \bar{\lambda}_{2}(n,k, D)-\bar{\lambda}_{1}(n, k,D) .  \]
	\end{theo}

   The above two theorems immediately give the following. 
 \begin{coro}
     		Let $\Omega$ be a strictly convex domain  in $\mathbb M_k$ ($k\geq 0$) with diameter $D<\pi/\sqrt{k}$. Let $\lambda_{i}(i=1,2)$ be the first two Dirichlet eigenvalues of the Schr\"odinger operator $-\Delta+V$ on $\Omega$. If $\Tilde{V}$ is a modulus of convexity of $V$ and (\ref{eqn:condition om lambda and V}) holds, then
		$$
\Gamma(\Omega)=\lambda_{2}-\lambda_{1} \geq \bar{\lambda}_{2}(n, k,D)-\bar{\lambda}_{1}(n,k, D) 
		$$
		where $\bar{\lambda}_{i}(n,k, D)(i=1,2)$ are the first two Dirichlet eigenvalues of the operator \eqref{one dim model} on $[-D/2,D/2]$.
  
 \end{coro}
	This fundamental gap estimate recovers the estimate in \cite{MR2784332}, and extends the estimate in \cite{MR3960269, MR4124117} to Schr\"odinger operators. 
 

The rest of this paper is structured as follows. Section 2 provides some necessary background materials for both Riemannian geometry and stochastic analysis. We define and study two diffusion couplings in Section 3, which will be used in our later proofs. Section 4 is devoted to the proof of the log-concavity of the first eigenfunction, and a proof for the gap comparison is given in Section 5.

\textbf{Acknowledgement}: We would like to thank Malik Tuerkoen and the referee for their careful reading of the paper and very helpful comments. 

	\section{Preliminaries}	
	
	\subsection{Brownian motion and diffusion in Euclidean space}
	In this subsection, we gather the essential materials from the classical stochastic analysis. These materials here are well-known; we include them so that readers without any probability background can acquire  intuitions and basic understandings. Some excellent and comprehensive presentations of these topics can be found in \cite{karatzas1991brownian, revuz2013continuous,MR2001996,MR3930614}.
	
	We fix a complete probability space $(A, \mathcal{F}, \mathbb{P})$. Here $A$ is the sample space,  $\mathcal F$ is a $\sigma$-field of subsets of $A$,  and $\mathbb P$  a complete probability measure on $(A, \mathcal{F})$. 

	\begin{defi}[Real Brownian motion] 
		A real-valued process $\{W_t\}_{t\geq 0}$ {{defined on $A$}} is said to be a standard Brownian motion if 
		\begin{enumerate}
			\item $W_0=0$.
			\item For each $t> 0$, $W_t$ is Gaussian with mean $0$ and variance $t$.
			\item $W_t$ has independent increments; that is, $W_t-W_s$ and $W_u-W_v$ are independent whenever $[s,t]\cap [u,v]=\emptyset$.
			\item $W_t$ is continuous as a function of $t\in [0, +\infty)$.
		\end{enumerate}
	\end{defi}
	More generally, for any $n\geq 1$, we can define a standard Brownian motion as a $\mathbb{R}^n$-valued process whose components are independent standard real Brownian motions.
	
	Let $\{\mathcal{F}_t\}_{t\geq 0}$ be a family of increasing sub $\sigma$-field of $\mathcal{F}$. A process $\left\{X_{t}\right\}_{t\geq 0}$ is called $\mathcal{F}_t$-adapted if $X_t$ is $\mathcal{F}_t$ measurable for any $t\geq 0$ (i.e. for any real numbers $a<b$, we have $\{\omega\in A:  a<X_t(\omega)<b \}\in \mathcal{F}_t$). 
	\begin{defi}
		A process $\left\{X_{t}\right\}_{t\geq 0}$ is said to be a $\mathcal{F}_t$-martingale if $X_{t}$ is $\mathcal{F}_t$-adapted and integrable for every $t \geq 0$, and the conditional expectation $\mathbb{E}\left(X_{t} \mid \mathcal{F}_{s}\right)$ is equal to $X_s$ almost surely, for all $0 \leq s \leq t$. In particular, $\mathbb{E}(X_t)=\mathbb{E}(X_0)$.
	\end{defi}

Brownian motion is the most classical example of martingales. The associated  $\mathcal{F}_t$ is given by the $\sigma$-field generated by $\{\omega\in A: a< W_s(\omega) <b, \}$,  where $0\leq s\leq t$ and $ a<b$ are any real numbers. We will call $\mathcal{F}_t$ the $\sigma$-field generated by $\{W_s\}_{0\leq s\leq t}$.

 For any $t>0$, a Brownian motion has unbounded variation on $[0,t]$. More precisely, 
	\[ \lim_{\abs{\Delta t_i}\rightarrow 0}\sum_{i}\abs{W_{t_{i+1}}-W_{t_i}}=+\infty  \]
	almost surely, where $\{t_i\}_{i\geq 1}$ is any partition of $[0,t]$ with mesh $\Delta t_i$. As a result, it is in general impossible to define integrals against Brownian motions as Riemann-Stieltjes integrals. To properly define stochastic integrals, we will need the so called quadratic variation, which we describe now.

     Suppose $X_t$ is a square-integrable martingale (i.e. $\mathbb{E}(X^2_t)<+\infty$ for all $t\geq 0$), the quadratic variation of $X_t$ is defined as
	\[ \langle X\rangle_t=\lim_{\abs{\Delta t_i}\rightarrow 0}\sum_{i}\abs{X_{t_{i+1}}-X_{t_i}}^2,  \]
    where the limit is taken in the probability measure $\mathbb{P}$. If two square integrable martingales $X_t, Y_t$ are adapted to the same filtration $\mathcal{F}_t$, then their linear combinations are also square integrable martingales. We define their cross variation as
	\[ \langle X, Y\rangle_t=\frac{1}{4}\left( \langle X+Y\rangle_t-\langle X-Y\rangle_t   \right). \]
	The process $\langle X\rangle_t$ is always of bounded variation. In particular, we can find a locally bounded function $h(X)$ such that $d\langle X\rangle_t=h(X)_tdt$. We can give another characterization of standard Brownian motions through quadratic variation.
	\begin{prop}[Levy's characterization]
		If $X_t$ is a continuous square integrable martingale with $\langle X \rangle_t=t$ such that  $\mathbb{P}(X_0=0)=1$, then $X_t$ is a standard real Brownian motion. 
	\end{prop}
	
	\begin{prop}\label{prop:Ito-integral}
		Let $Z_t$ be a $\mathcal{F}_t$-adapted process and $X_t$ be a square integrable $\mathcal{F}_t$-martingale such that
		\[ \int_{0}^{T}Z^2_sd\langle X\rangle_s<+\infty. \]
		Then, along any partitions of $[0,T]$, the limit
		\[ \lim_{\abs{\Delta t_i}\rightarrow 0}\sum_{i}Z_{t_i}(X_{t_{i+1}}-X_{t_i})  \]
		exists in $L^2(A)$. This limit will be denoted as $\int_{0}^{t}Z_sdX_s.$
	\end{prop}

We can now define stochastic integration.

\begin{defi}[It\^o's integral]
    The limit in Proposition~\ref{prop:Ito-integral} is called the It\^o's integral of $Z_t$ against $X_t$, and is denoted by $I_t(Z)=\int^{t}_{0}Z_s dX_s$. The quadratic variation and cross variation of It\^o's integral are given as
\begin{equation*}
    \langle I(Z) \rangle_t = \int_{0}^{t}Z^2_sd\langle X\rangle_s, \ \ \ \ \langle I(Z),I(W) \rangle_t=\int^{t}_{0} Z_s W_s d\langle X \rangle_s.
\end{equation*}
\end{defi}

	\begin{rema}
		It is important to note that in the previous approximating Riemann sum, we use the left point $t_i$ for $Z_s$ in each interval $[t_i,t_{i+1}]$. This special choice makes It\^o's integrals martingales.
	\end{rema}
	Another important type of stochastic integration is the following. 
	\begin{defi}
		The Stratonovich integral of  $Z_t$ against $X_t$ is defined to be
		\[  \int_{0}^{t} Z_s\circ dX_s := \int_{0}^{t} Z_s dX_s+\frac{1}{2}\langle Z, X \rangle_t. \]
	\end{defi}
	The change of variable formula for stochastic integration is drastically different from the classical calculus. 
 
	\begin{prop}[It\^o's formula]
		Let $f\in C^2(\mathbb{R}^n)$ and $\{X_t\}_{0\leq t\leq T}$ is a square integrable martingale. Then for any $0\leq t\leq T$,
		\begin{align}
			f(X_t)&=f(X_0)+\sum_{i=1}^{n}\int_{0}^{t}\partial_{i}f(X_s) dX^i_s+\frac{1}{2}\sum_{i,j=1}^{n}\int_{0}^{t}\partial^2_{i,j} f(X_s)d\langle X^i,X^j \rangle_s  \label{Ito's formula}  \\
		&=f(X_0)+\sum_{i=1}^{n}\int_{0}^{t}\partial_{i}f(X_s) \circ dX^i_s.  \nonumber
		\end{align}
	\end{prop}
    We see that in \eqref{Ito's formula}, there is an extra second derivative term, which does not appear in the classical calculus. Intuitively, this term corresponds to the second order term in the Taylor series expansion. Since martingales have finite quadratic variations, the sum of the second order terms in the Taylor expansion converges to a non-zero limit, which leads to the extra term in the It\^o's formula.


 Now we can formulate the next
	\begin{defi}
		\label{Diffusion}
		Let $X_t$ be a process on $\mathbb{R}^n$ and $L$ a smooth second order differential operator on $\mathbb{R}^n$. We say $X_t$ is the diffusion generated by $L$, if for any $f\in C_c^2(\mathbb{R}^n)$ the process
		\[ M^f_t:=f(X_t)-f(X_0)+\int_{0}^{t}Lf(X_s)ds  \]   
		is a martingale. Moreover, $L$ is called the generator of $X_t$.
	\end{defi}
	By It\^o's formula \eqref{Ito's formula}, it is easily checked that that Brownian motion is the diffusion generated by $\frac{1}{2}\Delta_{\mathbb{R}^n}$. 
	
	A canonical example of diffusion is provided by stochastic differential equations (SDE) of the form
	\begin{equation}
		\label{SDE_R^n}
		X_t=X_0+\sum_{i=1}^{n}\int_{0}^{t}V_i(X_s)\circ dW^i_s+\int_{0}^{t}V_0(X_s)ds,\ 0\leq t\leq T,\ X_0\in\mathbb{R}^n,
	\end{equation}
	where $\{V_i\}_{0\leq i\leq n}$ are smooth vector fields on $\mathbb{R}^n$ and $W_t$ is a $\mathbb{R}^n$-valued Brownian motion. The following existence and uniqueness result is well-known. 
	\begin{theo}
		\label{Existence and uniqueness in Rn}
		Suppose $\{V_i\}_{0\leq i\leq n}$ are Lipschitz smooth vector fields on $\mathbb{R}^n$. Then for any $T>0$, \eqref{SDE_R^n} admits a unique solution adapted to the Brownian filtration $\mathcal{F}_t$.
	\end{theo}
	In addition, the generator of $X_t$ defined above is given by the H\"ormander type operator $L=\frac{1}{2}\sum_{i=1}^{n}V_i^2+V_0$.
	
	Diffusion are also (strong) Markov processes. More specifically, we have
	\[ \mathbb{E}^{X_0}(X_{s+t} \mid \mathcal{F}_s)= \mathbb{E}^{X_s}(X_{t} ),\ \forall s,t\geq 0 \]
	Here the sup-scripts are used to indicate the starting points. We also emphasis that in the above equation $s,t$ can be replaced by stopping times adapted to the underlying filtration. Intuitively, it says that to predict the behavior of $X_{s+t}$ with the knowledge of $\{X_r,\ 0\leq r\leq s \}$ is the same as doing the prediction with knowledge of only $X_s$. In other words, for $X_t$, given the current position, the past and the future are independent. The Markov property of diffusion allows one to construct semi-groups from $X_t$. More specifically, for any bounded function $f\in C_c(\mathbb{R}^n)$
	\[ P_tf(x)=\mathbb{E}^x(f(X_t)) \] 
	is a strong continuous semi-group, whose generator coincides with the generator of $X_t$.
	\subsection{Coupling of diffusion in Euclidean space}
	In this section, we review some basic facts of the coupling method of diffusion in Euclidean space. Let $Z_t$ be a diffusion generated by a given elliptic operator $L$. We have the following
	\begin{defi}
		A couple diffusion $(X_t,Y_t)\in \mathbb{R}^n\times \mathbb{R}^n\cong \mathbb{R}^{2n}$ is said to be a coupling of $Z_t$, if both $X_t$ and $Y_t$ are diffusion generated by $L$.
	\end{defi}
	\begin{rema}
		This definition might look weird at first since one would have thought that we should have $(X_t,Y_t)=(Z_t,Z_t)$. However, the uniqueness we claimed for $Z_t$ is in the sense of probability distribution. In other words, if $X_t,Y_t,Z_t$ have the same starting point, then for any Borel set $B\in \mathbb{R}^n$, we always have
		\[ \mathbb{P}(X_t\in B)=\mathbb{P}(Y_t\in B)=\mathbb{P}(Z_t\in B),\ \forall t>0  \]
		But this does not imply $X_t=Y_t=Z_t$. We can give a simple example. Let $W_t$ be a standard Brownian motion on $\mathbb{R}^1$, then by symmetry $(W_t, -W_t)$ is a coupling of real Brownian motion. But clearly, $W_t$ and $-W_t$ are not identical.
	\end{rema} 
	
	For our purpose, we will consider
	\begin{align*}
		X_t&=X_0+\sum_{i=1}^{n}\int_{0}^{t}V_i(X_s)\circ dW^i_s+\int_{0}^{t}V_0(X_s)ds\\
		Y_t&=Y_0+\sum_{i=1}^{n}\int_{0}^{t}V_i(Y_s)\circ dB^i_s+\int_{0}^{t}V_0(Y_s)ds, 
	\end{align*}
	where $(W_t, B_t)$ is a coupling of $\mathbb{R}^n$ Brownian motions. Obviously, $(X_t, Y_t)$ is a coupling of diffusion generated by $L=\frac{1}{2}\sum_{i=1}^{n}V_i^2+V_0$. So, each coupling $(W_t, B_t)$ of $\mathbb{R}^n$ Brownian motions gives us a coupling $(X_t, Y_t)$. A special type of coupling that has attracted a lot of attentions is the so called Kendall-Cranston  coupling, or reflection coupling, which was formally introduced in \cite{MR1120916} and we now introduce. 
	
	We consider the orthogonal matrix
	\[m'(x,y)=I_n-2\frac{(x-y)(x-y)^*}{\abs{x-y}^2},\ x,y\in \mathbb{R}^n, x\neq y.   \]
	Multiplying $m'(x,y)$ with an vector $v$ gives the reflection of $v$ across the hyperplane perpendicular to the line segment that connects $x$ and $y$.
	
	Let $W_t$ be a standard $\mathbb{R}^n$ Brownian motion, and define
	\begin{align}
		X_t&=X_0+\sum_{i=1}^{n}\int_{0}^{t}V_i(X_s)\circ dW^i_s+\int_{0}^{t}V_0(X_s)ds  \label{X_t} \\
		Y_t&=Y_0+\sum_{i=1}^{n}\int_{0}^{t}V_i(Y_s)\circ dB^i_s+\int_{0}^{t}V_0(Y_s)ds\\
		B_t&=\int_{0}^{t}m'(X_t, Y_t) dW_t.  \label{B_t} 
	\end{align}
	Since $m'(X_s,Y_s)$ is orthogonal, by Levy's characterization, $B_t$ is another Brownian motion adapted to $\sigma(W_t)$. Hence, the above-defined $(X_t, Y_t)$ is a coupling. 
We will refer to $\tau:=\inf\{t\geq 0:\ X_t=Y_t \}$ as the coupling time. Now, the complete reflection coupling is given by \eqref{X_t} - \eqref{B_t} and $X_t=Y_t, \forall t\geq \tau.$
	Notice that we set $X_t=Y_t$ after the coupling time. By the strong Markov property, this does not change the generator of $X_t$ and $Y_t$. The fact that these two diffusion stick together when $t$ is sufficiently large is crucial to our analysis later.
	\subsection{Diffusion on Riemannian manifolds}
	Let $M^n$ be a smooth compact Riemannian manifold of dimension $n>1$ without boundary. We first make sense of SDE of the form
	\begin{equation}
		\label{SDE on manifold}
		X_t=X_0+\sum_{i=1}^{n}\int_{0}^{t}V_i(X_s)\circ dW^i_s+\int_{0}^{t}V_0(X_s)ds,\ 0\leq t\leq T,\ X_0\in M^n,
	\end{equation}
	where $\{V_i\}_{0\leq i\leq n}$ are smooth vector fields on $M^n$ and $W_t$ is a $\mathbb{R}^n$-valued Brownian motion.
	\begin{defi}
		An adapted process $X_t$ is said to be the solution to \eqref{SDE on manifold} if for any $f\in C^\infty(M^n)$, we have
		\[ f(X_t)=f(X_0)+\sum_{i=1}^{n}\int_{0}^{t}V_if(X_s)\circ dW^i_s+\int_{0}^{t}V_0f(X_s)ds,\ 0\leq t\leq T. \]
	\end{defi}
	By Whitney's embedding Theorem, $M^n$ can be smoothly embeded into $\mathbb{R}^{2n}$ as a closed subset. We can also extend $\{V_i\}_{0\leq i\leq n}$ to smooth vector fields $\{\hat{V}_i\}_{0\leq i\leq n}$ on $\mathbb{R}^{2n}$. The existence and uniqueness of solution to \eqref{SDE on manifold} is guaranteed by Theorem \ref{Existence and uniqueness in Rn} and the observation that $X_t$ as a process in $\mathbb{R}^{2n}$ stays in $M^n$. 
	
	Diffusion on Riemannian manifolds is defined as a natural extension of the Definition~\ref{Diffusion}.
	\begin{defi}
		Let $X_t$ be a process on $M^n$ and $L$ a smooth second order differential operator on $M^n$. We say $X_t$ is the diffusion generated by $L$, if for any $f\in C^2(M^n)$ the process
		\[ M^f_t:=f(X_t)-f(X_0)+\int_{0}^{t}Lf(X_s)ds  \]   
		is a martingale.  
	\end{defi}
	We are ready for the next
	\begin{defi}
		A process $X_t$ on $M^n$ is said to be a Brownian motion if the generator of $X_t$ is $\frac{1}{2}\Delta$, where $\Delta$ is the Laplace-Beltrami operator on $M^n$.
	\end{defi}
	In order to take full advantage of stochastic calculus, we would like to write a Brownian motion on $M^n$ as the solution to a SDE on $M^n$ (note that this task is trivial on $\mathbb{R}^n$). Thus, it is desirable to write the Laplace-Beltrami operator as a H\"ormander type operator $L=\frac{1}{2}\sum_{i=1}^{n}V_i^2+V_0$. However, there is no intrinsic way of achieving this on a general Riemannian manifold. We will employ the Eells-Elworthy-Malliavin approach (see chapter 3 of \cite{Hsu}), which we now describe. 
	
	Let $O(M^n)$ be the orthonormal frame bundle of $M^n$ and $\pi:O(M^n)\rightarrow M^n$ the canonical projection. For any $u\in O(M^n)$, we denote $H_{e_i}(u)=H_i(u)$ the horizontal vector field such that $\pi_*H_i(u)=u\cdot e_i$, 
where $\{e_i\}_{1\leq i\leq n}$ is a basis of $\mathbb{R}^n$. Bochner's horizontal Laplacian on $O(M^n)$ is defined as $\Delta_{O(M^n)}=\sum_{i=1}^{n}H^2_i$. In particular, we have
	\begin{theo}
		Let $f\in C^\infty(M^n)$ and $\Tilde{f}=f\circ \pi$ its lift to {$O(M^n)$}, then for any $u\in O(M^n)$
		\[\Delta_{O(M^n)}\Tilde{f}(u)=\Delta f(\pi(u)).  \]
	\end{theo}
	With Bochner's horizontal Laplacian, a horizontal Brownian motion on $O(M^n)$ can be written as the solution to 
	\begin{equation*}
		\alpha_t=\alpha_0+\sum_{i=1}^{n}\int_{0}^{t}H_i(\alpha_s)\circ d B^i_s,\ \alpha_0\in O(M^n).
	\end{equation*}
	Its projection $X_t=\pi(\alpha_t)$ is a Brownian motion on $M^n$ starts at $\pi(\alpha_0)$. 
 
	\subsection{Coupling of Brownian motions on Riemannian manifolds}
	The definition of coupling on Riemannian manifold is exactly the same as on $\mathbb{R}^n$ as we introduced before. In order to construct a reflection coupling of Brownian motions on $M^n$, we only need to replace the orthogonal matrix $m'(x,y)$ by the so called mirror map $m(x,y):T_xM^n\rightarrow T_yM^n$. For each $w\in T_xM^n$, $m(x,y)(w)$ is obtained by first parallel transport $w$ along the unique minimizing geodesic from $x$ to $y$, then reflect the resulting vector with respect to the hyperplane perpendicular to the geodesic at $y$. Clearly, $m(x,y)$ is an isometry and generalization of $m'(X,Y)$ that we used in $\mathbb{R}^n$. It should be pointed out that if $x,y$ are conjugate points then $m(x,y)$ is undefined. This issue could be handled by an extension method (see \cite[Theorem 6.62]{Hsu}).
	
	The reflection coupling of Brownian motions on $M^n$ is defined as
	\begin{align*}
		d U_t&=\sum_{i=1}^{n}H_i(U_t)\circ d B^i_t,\ \ \ \ \  \ \ \ \ \ \ 
		d V_t=\sum_{i=1}^{n}H_i(V_t)\circ d W^i_t,  \\
		X_t&=\pi(U_t), \ \ \ \ \ \ 
		Y_t=\pi(V_t),\ \ \ \ \ 
		dW_t=V_tm(X_t,Y_t)U_t^{-1}dB_t.
	\end{align*}
	
	\subsection{Boundary control of solutions of heat equation} 
	
    We recall from \cite[Lemma 3.4]{MR3960269} a general near boundary control of positive functions on a convex domain that vanish on the boundary .
	
	\begin{prop}\label{prop:key}
 Let $\Omega$ be a uniformly convex bounded domain in a Riemannian manifold $M^n$,
and $u:  \overline{ \Omega}  \times \mathbb{R}_+ \to \mathbb{R}$ a $C^2$ function such that $u$ is positive on $\Omega$, $u(\cdot, t) = 0$ and  $\nabla u  \neq 0$ on $\partial \Omega$.   Given $T<\infty$,  there exists $r_1> 0$ such that $\Hess \log u|_{(x,t)} < 0$ whenever $d(x,\partial \Omega) < r_1$ and $t\in [0,T]$, and $C \in\mathbb{R}$ such that $\Hess\, \log u|_{(x,t)}(v,v) \leq C \|v\|^2$ for all $x \in \Omega$ and $t \in [0,T]$.
	\end{prop}
	
		We will also need the next
		\begin{prop}\label{Solutions to heat equations}
			Let $\Omega$ be a bounded strictly convex domain in a Riemannian manifold with smooth boundary $\partial \Omega$. Denote by $\rho_{\partial \Omega}: \bar{\Omega} \rightarrow \mathbb{R}_{+}$the distance function to the boundary $\partial \Omega$, and $N$ the unit inward normal vector field on $\partial \Omega$.
			Let $u_{0} \in C^{\infty}(\bar{\Omega})$ be positive in $\Omega$ such that $u_{0}=0$ and $\nabla u_{0} \neq 0$ on $\partial \Omega$. Let $u: \mathbb{R}_{+} \times \bar{\Omega} \rightarrow \mathbb{R}_{+}$be a smooth solution to
			$$
			\begin{aligned}
				\frac{\partial u}{\partial t} & =(\Delta-V) u \quad \text { in } \mathbb{R}_{+} \times \Omega ; \\
				u & =0 \quad \text { on } \mathbb{R}_{+} \times \partial \Omega \text { and } u(0, \cdot)=u_{0} \text { in } \bar{\Omega} .
			\end{aligned}
			$$
			Then the solution $u$ verifies that
			
			(1) for any $T>0, \theta_{T}:=\inf _{[0, T] \times \partial \Omega}|\nabla u|>0$;
			
			(2) for every $x \in \partial \Omega, \nabla u(t, x)=|\nabla u(t, x)| N(x)$; 
			
			(3) $\lim _{\rho_{\partial \Omega}(x) \rightarrow 0} \frac{u(t, x)}{|\nabla u(t, x)| \rho_{\partial \Omega}(x)}=1$ uniformly in $t \in[0, T]$;
			
			(4) for any $T>0$, there exists $C_{T} \geq 0$ such that $\left.\Hess \log u\right|_{(t, x)}(y, y) \leq C_{T}|y|^{2}$ for all $t \in[0, T], x \in \Omega$ and $y \in 	T{\Omega}$ tangent vector at $x$. 
			
		\end{prop}
		\begin{proof}
			The first statement is a consequence of maximum principle for an elliptic operator and non-constancy of $u_0$. Since $u$ behaves as a level set on $ \mathbb{R}_{+} \times \partial \Omega $, the second statement follows. The third statement follows from the following identification: for $r>0$, we write $\partial_{r} \Omega=\left\{x \in \Omega: \rho_{\partial \Omega}(x) \leq r\right\}$ for the $r$-neighborhood of $\partial \Omega$. As \cite{MR2158015}, there exists $r_{0} \in(0, D / 2)$ such that $\rho_{\partial \Omega}$ is smooth on $\partial_{r_{0}} \Omega$. Then for any $x \in \partial_{r_{0}} \Omega$, there exists a unique $x^{\prime} \in \partial \Omega$ such that $\rho_{\partial \Omega}(x)=\left|x-x^{\prime}\right|$ (meaning that the minimal geodesic distance between $x$ and $x'$) and $\nabla \rho_{\partial \Omega}(x)=N\left(x^{\prime}\right)$. In particular, $\nabla \rho_{\partial \Omega}=N$ on the boundary $\partial \Omega$. The last statement follows from Proposition~\ref{prop:key}. 
		\end{proof}

		\subsection{Jacobi fields and Variations in the Model Spaces }
  Here we recall the notations and variations from \cite{MR3960269} which we will use later. 
  
  Let $\mathbb{M}_{K}^{n}$ be the  $n$-dimensional simply connected manifold with constant sectional curvature $K$. 
 Let $\operatorname{sn}_{K}(s)$ be the solution of
	$$
	\operatorname{sn}_{K}^{\prime \prime}(s)+K \operatorname{sn}_{K}(s)=0, \quad \operatorname{sn}_{K}(0)=0, \quad \operatorname{sn}_{K}^{\prime}(0)=1, 
	$$
	$\operatorname{cs}_{K}(s)=\operatorname{sn}_{K}^{\prime}(s)$,  and $\operatorname{tn}_{K}(s)=K \frac{\operatorname{sn}_{K}(s)}{\operatorname{cs}_{K}(s)}=-\frac{\operatorname{cs}_{K}^{\prime}(s)}{\operatorname{cs}_{K}(s)}$. 
 Explicitly we have
	
	$$
	\begin{gathered}
		\operatorname{sn}_{K}(s)= \begin{cases}\frac{1}{\sqrt{K}} \sin (\sqrt{K} s), & K>0 \\
			s, & K=0 \\
			\frac{1}{\sqrt{-K}} \sinh (\sqrt{-K} s) & K<0\end{cases} \\
		\operatorname{cs}_{K}(s)= \begin{cases}\cos (\sqrt{K} s), & K>0 \\
			1, & K=0 \\
			\cosh (\sqrt{-K} s), & K<0\end{cases}
	\end{gathered}
	$$
	
	and
	
	\begin{equation}
	\operatorname{tn}_{K}(s)= \begin{cases}\sqrt{K} \tan (\sqrt{K} s), & K>0 \\ 0, & K=0 \\ -\sqrt{-K} \tanh (\sqrt{-K} s). & K<0\end{cases}  \label{tn_k}
	\end{equation} 
	
		Let $x_0 \neq y_0$ in $\Omega$ with $\rho(x_0,y_0) =d_0$. Let $\gamma(s): \left[-\tfrac{d_{0}}{2}, \tfrac{d_{0}}{2}\right] \mapsto \Omega$ be the normal minimal geodesic connecting $x_0$ to $y_0$.   Associated to a vector $v \oplus w \in T_{x_{0}} \Omega \oplus T_{y_{0}} \Omega$, we use the variation $\eta(r, s)$ from \cite[Page 362]{MR3960269}.  Let $\sigma_{1}(r)$ be the geodesic with $\sigma_{1}(0)=x_{0}, \frac{\partial}{\partial r} \sigma_{1}(0)=v, \sigma_{2}(r)$ be the geodesic with $\sigma_{2}(0)=y_{0}, \frac{\partial}{\partial r} \sigma_{2}(0)=w$, and $\eta(r, s), s \in\left[-\frac{d_{0}}{2}, \frac{d_{0}}{2}\right]$, be the minimal geodesic connecting $\sigma_{1}(r)$ and $\sigma_{2}(r)$, with $\eta(0, s)=\gamma(s)$. Namely $\eta(r, s)=\exp _{\sigma_{1}(r)} s V(r)$ for some $V(r)$. Since we are in a strictly convex domain, every two points are connected by a unique minimal geodesic, the variation $\eta(r, s)$ is smooth. Denote the variation field $\frac{\partial}{\partial r} \eta(r, s)$ by $J(r, s)$. Then $J(r, s)$ is the Jacobi field along $s$ direction satisfying $J\left(r,-\frac{d_{0}}{2}\right)=v, J\left(r, \frac{d_{0}}{2}\right)=w$. Denote $J(s)=J(0, s)$. Note that with this parametrization, in general, for fixed $r, \eta(r, s)$ is not unit speed when $s \neq 0$. 

 Denote  the unit vector of $\frac{\partial}{\partial s} \eta(r, s)$ by \begin{equation}
    T(r, s)=\frac{\eta^{\prime}}{\left\|\eta^{\prime}\right\|}. \label{T} 
 \end{equation}

		
		
		For space with constant sectional curvature $\mathbb{M}_{K}^{n}$, it is shown in \cite{MR3960269} that
            \begin{align}
            &\left.\nabla_{r} T(r, s)\right|_{r=0}=J^{\prime}(s), \label{Jacobi field first derivative}\\
            &\left.\nabla_{r} \nabla_{r} T\right|_{r=0}=-\left\|J^{\prime}(s)\right\|^{2} e_{n}. \label{Jacobi field second derivative}
            \end{align}

Let $\{e_i\}, i = 1, \dots, n$ be an orthonormal basis of $T_{x_0}$ with $e_n = \gamma'(0)$, parallel translate $e_i$ along $\gamma$ to $y_0$. 
  In $\mathbb M^n_k$, we have explicit formulas for the Jacobi fields with given boundary values. 
  For each $1 \leq i \leq n-1$, with boundary values $0$ and $ e_{i}$, the Jacobi fields are $Q_{i}(s)=\frac{\operatorname{sn}_{K}\left(\frac{d_{0}}{2}+s\right)}{\operatorname{sn}_{K}\left(d_{0}\right)} e_{i}(s)$; for boundary values $e_{i}$ and $0$, $Q_{i}(s)=\frac{\operatorname{sn}_{K}\left(-s+\frac{d_{0}}{2}\right)}{\operatorname{sn}_{K}\left(d_{0}\right)} e_{i}(s)$;  and for  $e_{i}$ and $e_i$,  $Q_{i}(s)=\frac{\operatorname{cs}_{K}\left(s\right)}{\operatorname{cs}_{K}\left(\frac{d_{0}}{2}\right)} e_{i}(s)$.

		\section{Diffusion couplings and related properties}
		In this section, we introduce two different couple diffusion and derive some useful properties for our use later. We first consider the following  $(X_t, Y_t)$ on the product manifold $M^n\times M^n$. This coupling will be used to prove the gap comparison. This coupling was also used in the Euclidean case in \cite{MR3489850}. We consider
		\begin{align} 
			&dU_t=\sum_{i=1}^{n}\sqrt{2}H_{e_i}(U_t)\circ dB^i_t+2H_{\nabla\log \phi_1}(U_t)dt,  \label{couping1} \\
			&dV_t=\sum_{i=1}^{n}\sqrt{2}H_{e_i}(V_t)\circ dW^i_t+2H_{\nabla\log \phi_1}(V_t)dt,\\
			&dW_t=V_tm_{X_tY_t}U_t^{-1} dB_t,\\
			&X_t=\pi(U_t),\ \ \ \ \ \ Y_t=\pi(V_t), \label{coupling2}
		\end{align} 
		where $m(X_t,Y_t)$ is the mirror map defined in Section 2.4 and $H_{e_i}$ is the horizontal vector field we defined in section 2.3. It is easily checked that $(X_t,Y_t)$ is a diffusion coupling in the sense of Kendall-Cranston \cite{kendall,Hsu}. 
        \begin{rema}
            The main motivation behind this definition is that both $X_t, Y_t$ are generated by the operator $L$ that appears in the ground state equation \eqref{Ground state equation} in section 5 below, which is essential in proving the gap comparison. It should be pointed out that with the Neumann boundary condition, $\phi_1$ is always constant. Therefore, $\nabla\phi_1=0$ and the coupling $(X_t,Y_t)$ given above becomes a coupling of Brownian motions, which is the canonical diffusion to study the first non-zero Neumann eigenvalue (see \cite{Hsu} chapter 6).  
        \end{rema}
  
        Our first task is to prove that for arbitrary starting points $X_0,Y_0\in \Omega$, the processes $X_t,Y_t$ will stay in $\Omega$ for all $t\geq 0$.
		\begin{lemm}\label{lem:process}
		Let $\Omega$ be a bounded strictly convex domain in a Riemannian manifold with smooth boundary $\partial \Omega$.	If $X_0, Y_0\in\Omega$, then $X_t,Y_t\in\Omega$ for all $t\geq 0$.
		\end{lemm} 
		\begin{proof}
			First we choose a smooth function $r:\bar{\Omega}\rightarrow \mathbb{R}$ such that $r(x)>0$ for all $x\in \Omega$ and for some $d_0\in (0, D)$,
			\[ r(x)=\rho_{\partial \Omega }(x),\ \forall x\in\{ p\in\Omega \mid \rho_{\partial \Omega }(p)<d_0  \}, \]
            where $\rho_{\partial \Omega }(x)$ denotes the geodesic distance between $x$ and the boundary $\partial \Omega$. By our definition and compactness (see for instance, Theorem 2.2 of \cite{MR2158015}), it is always possible to find $c_0>0$, such that $\Delta r(x)\geq -c_0$ for all $x\in\Omega$. 
            
            Since we only need to know the behavior of $X_t$ near the boundary, it suffices to show that $r(X_t)>0$ for all $t\geq 0$. By the same arguments in \cite[Theorem 3]{MR1120916} or \cite[Theorem 6.6.2]{MR3197874}, we have 
			\[ dr(X_t)=\sqrt{2}d\beta_t+\left[\Delta r(X_t)+2\nabla\log \phi_1(r)(X_t)\right]dt,   \]
			where $\beta_t$ is a standard Brownian motion. Since $e^{-\lambda_1t}\phi_1$ solves the heat equation in Proposition~\ref{Solutions to heat equations},  $\phi_1(x)$ verifies (2) (3) of Proposition~\ref{Solutions to heat equations}. Hence, we can find $d\in(0, d_0)$ such that
			\[  \Delta r(x)+2\nabla\log \phi_1(r)(x)\geq 2\nabla\log \phi_1(r)(x)-c_0\geq \frac{1}{r(x)},   \]
			for all $x\in \Omega$ with $\rho_{\partial \Omega(x)}<d$. So we get, when $\rho_{\partial\Omega}(X_t)<d$,
			\[  dr(X_t)>\sqrt{2}d\beta_t+\frac{1}{r(X_t)}dt,  \]
			which implies $r(X_t)>0$. Our proof is complete.
		\end{proof}
		
		Recall that $\tau$ is the coupling time of $(X_t,Y_t)$. Our next result shows that $\tau$ is finite almost surely. 
		\begin{lemma}\label{lem:successful coupling}
		Let $(M^n,g)$ be a complete Riemannian manifold with Ricci curvature lower bound $(n-1)k, k\in \mathbb{R}$ and $\Omega\subset M^n$ be a convex domain with diameter $D>0$. Then, for any starting points $X_0,Y_0\in \Omega$, the coupling $(X_t,Y_t)$ is successful in the sense that $\tau<\infty$ almost surely. 
		\end{lemma}
		\begin{proof}
			Let $\rho_t=\rho(X_t,Y_t)$ be the geodesic distance between $X_t,Y_t$. It suffices to prove $\rho_t$ hits $0$ in finite time almost surely. Let $\{e_i\}_{1\leq i\leq n}$ be an orthonormal frame at $X_t$ and we parallel transport them along the unique minimizing geodesic that connects $X_t, Y_t$. We further assume that $e_n$ is the tangential direction. By It\^o's formula (see section 6.6 of \cite{Hsu} for a similar computation)
			\begin{equation*}
				d\rho_t=2\sqrt{2}d\beta_t+\left(\sum_{i=1}^{n}I(Q_i,Q_i)+2\langle \nabla\log\phi_1(Y_t), \gamma'(Y_t)  \rangle -2\langle \nabla\log\phi_1(X_t), \gamma'(X_t)  \rangle \right)dt,
			\end{equation*}
			where $\beta_t$ is a Brownian motion independent from $B_t$ and $Q_i$ is the Jacobi field associated with any variation whose tangent vector at $(X_t,Y_t)$ is given by $e_i\oplus m(X_t,Y_t)e_i$. The index form of a vector field $X$ (with not necessarily vanishing endpoints) along a geodesic $\gamma$ is defined by
			\begin{align*}
				I (\gamma,X,X): = \int_0^T \left( \langle \nabla_{\gamma'}  X, \nabla_{\gamma'} X  \rangle-\langle R(\gamma',X)X,\gamma'\rangle \right) dt,
			\end{align*}
			where $\nabla$ is the Levi-Civita connection and $R$ is the Riemann curvature tensor of a given manifold. 
			We will use Lemma 1 from \cite{MR0841588} to verify that $\rho_t$ hits $0$ in finite time. We have
			\begin{align*}
				\theta_t:&=2\left(\langle \nabla\log\phi_1(Y_t), \gamma'(Y_t)  \rangle -\langle \nabla\log\phi_1(X_t). \gamma'(X_t)  \rangle \right)+\sum_{i=1}^{n}I(Q_i,Q_i)\\
				&=2\int_{0}^{\rho_t}\nabla^2\log\phi_1(\gamma'(s),\gamma'(s))ds+\sum_{i=1}^{n}I(Q_i,Q_i).
			\end{align*}
			By Index lemma, we can find $C_1(k)>0$ such that
			\[ \frac{1}{2}\sum_{i=1}^{n}I(Q_i,Q_i)\leq -(n-1)tn_k(\rho_t/2)\leq C_1. \]
			Consequently 
			\[ \theta_t\leq 2\int_{0}^{\rho_t}\nabla^2\log\phi_1(\gamma'(s),\gamma'(s))ds+C_1. \]
			Let
			\[ C_2=\sup\{  \Hess \phi_1(p)(v,v)  :   p\in\bar{\Omega},\  v\in T_pM^n,\ \vert v\vert=1 \}.  \] 
   By (4) of  Proposition~\ref{Solutions to heat equations}, $C_2$ is finite. 
			We introduce
			\begin{equation*}
				\eta(r)=\begin{cases}
					\frac{1}{4}C_2 \cdot r+C_1 &  r\leq D\\
					-1 & r>D.
				\end{cases}
			\end{equation*}
			Since $X_t,Y_t$ never leave $\Omega$, $\rho_t\leq D$ for all $t\geq 0$ and we have $\theta_t\leq 8\eta(\rho_t) $ for all $t\geq 0$. It follows that the assumptions of Lemma 1 of \cite{MR0841588} are verified. We conclude that $\rho_t$ hits $0$ in finite time almost surely. The proof is complete.
		\end{proof}
		Next we introduce another important pair of diffusion $(X'_t,Y'_t)$, which will be used in establishing the log-concavity estimate for the first eigenfunction of $-\Delta+V$. Let us consider 
		\begin{align}\label{eq:new-couple-diffusion}
			&dU'_t=\sum_{i=1}^{n}\sqrt{2}H_{e_i}(U'_t)\circ dB^i_t+2H_{\nabla\log \phi_1}(U'_t) dt+2tn_k\left(\frac{\tilde{\rho}(U'_t,V'_t)}{2}\right)\tilde{\gamma}'(t)dt,\\ \nonumber
			&dV'_t=\sum_{i=1}^{n}\sqrt{2}H_{e_i}(V'_t)\circ dW^i_t+2H_{\nabla\log \phi_1}(V'_t) dt-2tn_k\left(\frac{\tilde{\rho}(U'_t,V'_t)}{2}\right)\tilde{\gamma}'(t)dt,\\ \nonumber
			&dW_t=V'_tm_{X_tY_t}(U'_t)^{-1} dB_t,\\ \nonumber
			&X'_t=\pi(U'_t),\ \ \ \ \ \ Y'_t=\pi(V'_t). 		
		\end{align} 
		Here $\tilde{\rho}$ is the lift of $\rho$, which is defined as $\tilde{\rho}(U,V)=\rho\left(\pi(U),\pi(V)\right)$, and $\tilde{\gamma}$ is the horizontal lift of the minimizing geodesic that goes from $X'_t$ to $Y'_t$. We note that $X'_t, Y'_t$ can not be defined separately since the drift term depends on both $X'_t$ and $ Y'_t$. For this reason, $(X'_t, Y'_t)$ is not a coupling of diffusion. We define $\tau'=\inf_{t\geq 0}\{ \rho(X'_t, Y'_t)=0$\}. With a slight abuse of terminology, we will still call $\tau'$ the coupling time. As before, we define $X'_t=Y'_t$ when $t\geq \tau'$. 
    
    \begin{rema}
        There is no intuitive explanation of the extra term, compared to $(X_t,Y_t)$, added in the definition of $(X'_t,Y'_t)$; it was what came out of our computations. However, it is interesting to note that when $k=0$, the couple diffusion $(X'_t,Y'_t)$ coincides with the coupling $(X_t, Y_t)$. We can thus regard $(X'_t,Y'_t)$ as a non-trivial generalization of $(X_t,Y_t)$ that has incorporated more information about the curvature of the underlying manifold. 
    \end{rema}

		One can argue in the same way as we did in Lemma~\ref{lem:process} to justify the following.
		\begin{lemm}
        Under the same assumptions of Lemma~\ref{lem:process}, if $X'_0, Y'_0\in\Omega$, then $X'_t,Y'_t\in\Omega$ for all $t\geq 0$.
		\end{lemm}
		
		Let us observe that, $(X_t', Y_t')$, although do not form a usual coupling, will also meet in finite time. In particular, we can give a comparison between $\tau$ and $\tau'$ based on $k$.
		\begin{lemma}\label{lem:successful coupling for modified diffusion}
			As in the setting of Lemma~\ref{lem:successful coupling}, we have $\tau'<\infty$ almost surely. Moreover, we have $\tau'\leq\tau$ when $k\geq 0$ and $\tau> \tau'$ when $k < 0$ almost surely.
		\end{lemma}
		\begin{proof}
			Since $\nabla_{\gamma'(X'_t)}\rho(X'_t,Y'_t)=-\nabla_{\gamma'(Y'_t)}\rho(X'_t,Y'_t)=-1$ , we have by It\^o's formula
            \begin{align*}
				d\rho(X'_t,Y'_t)= 2\sqrt{2}d\beta_t+\left(\sum_{i=1}^{n}I(Q_i,Q_i)+F(X'_t,Y'_t)-4tn_k\left(\frac{\rho(X'_t,Y'_t)}{2}\right) \right)dt,
			\end{align*}
             where
			\[ F(X_t,Y_t)= \langle \nabla\log\phi_1(Y_t), \gamma'(Y_t)  \rangle -\langle \nabla\log\phi_1(X_t), \gamma'(X_t)  \rangle, \]
            and $\beta_t$ is a standard Brownian motion. Since $tn_k\left(\rho(X'_t,Y'_t)/2\right)$ is uniformly bounded in $\Omega$, the exact same argument of Lemma \ref{lem:successful coupling} applies to $\rho(X'_t,Y'_t)$ and $\tau'<+\infty$ follows.
		
            For our second claim, we observe that $tn_k\left(\rho(X'_t,Y'_t)/2\right)\geq 0$ when $k\geq 0$ and $tn_k\left(\rho(X'_t,Y'_t)/2\right)< 0$ when $k< 0$. Hence  from the diffusion comparison theorem we deduce that $\rho(X_t,Y_t)\geq \rho(X'_t,Y'_t)$ when $k\geq 0$ and $\rho(X_t,Y_t)<\rho(X'_t,Y'_t)$ when $k<0$. The proof is complete. 
            \end{proof}
		
		\section{Log-concavity of the first eigenfunction}
		In this section we establish a log-concavity estimate of the first eigenfunction of the Schr\"odinger operator $-\Delta+V$ on $\Omega$, Theorem~\ref{log-concavity}. Our analysis in this section rely on the diffusion $(X'_t,Y'_t)$ that we introduced in \eqref{eq:new-couple-diffusion}. 
        
        We pick two distinct points $x,y\in \Omega$ and let them be the starting points of $X'_t,Y'_t$ respectively. Denote $\xi_t=\rho(X'_t,Y'_t)/2$.  Let $\gamma$ be the normal minimal geodesic that goes from $X'_t$ to $Y'_t$ with $\gamma(-\xi_t) = X_t'$ and $\gamma(\xi_t) = Y_t'$. We choose $\{e_i\}_{1\leq i\leq n}$ an orthonormal basis at $X_t$ with $e_n=\gamma'(X'_t)$ and parallel transport $\{e_i\}_{1\leq i\leq n}$ along $\gamma$. Let 
		\[ E_i= e_i\oplus e_i,\ 1\leq i\leq n-1,\ E_n=e_n\oplus -e_n.  \]
		We define 
		\begin{equation}
			\label{F_s}
			F_t=\langle \nabla\log \phi_1(Y'_t), \gamma'(Y'_t)  \rangle -\langle \nabla\log \phi_1(X'_t), \gamma'(X'_t)  \rangle.
		\end{equation}
		One crucial step of our method is to find the SDE of $F_t$. The essence is to use It\^o's formula and the geodesic variations as in \cite{MR3960269} (see Section 2.6) to compute the first two orders of covariant derivatives.


    The following is our key formula. 
		\begin{prop}
			Let $\lambda_1$ be the first eigenvalue of $-\Delta+V$ on a convex domain $\Omega \subset \mathbb M^n_k$ and $\omega=\log \phi_1$. Then the SDE of $F_t$ is given by
			\begin{align}
				dF_t&=d\{\text{martingale}\}\nonumber\\
				&+\left\{ \langle \nabla V(Y'_t), e_n\rangle-\langle \nabla V(X'_t), e_n\rangle\right\}dt \nonumber\\
				& +(n-1)(K-tn^2_k(\xi_t))\left\{\langle \nabla \omega(Y'_t), e_n\rangle-\langle \nabla \omega(X'_t), e_n\rangle \right\}dt\nonumber\\
				& +2tn_k(\xi_t)\left[\langle  \nabla\omega(Y'_t), e_n  \rangle^2+\langle \nabla\omega(X'_t), e_n  \rangle^2 +2\lambda_1-V(X_t')-V(Y_t')\right]dt\nonumber\\
				&+\frac{2}{sn_k(2\xi_t)}\sum_{i=1}^{n-1}\left(\langle  \nabla\omega(Y'_t), e_i  \rangle-\langle \nabla\omega(X'_t), e_i  \rangle \right)^2dt.\label{dF for log-concavity}
			\end{align}
			\end{prop}
            \begin{rema}
                It is possible to give the precise expression for the martingale part of $dF_t$, which, however, is irrelevant to our purpose. We write it as it is now to save space.
            \end{rema}
			\begin{proof}
				Recall we use the notation $\xi_t=\rho(X'_t,Y'_t)/2$. By It\^o's formula \eqref{Ito's formula}, we have 
				\begin{align*}
					dF_t&=\sum_{i=1}^{n}\nabla_{e_i\oplus 0} F_t(X'_t, Y'_t)\circ d(X'_t)^i+\nabla_{0\oplus e_i} F_t(X'_t, Y'_t)\circ d(Y'_t)^i\\
					&= d\{\text{martingale}\}+\{\sum_{i=1}^{n}\nabla_{E_i}\nabla_{E_i}F_t(X'_t, Y'_t) \} dt\\
					&+\left\{2\nabla\log \phi_1(X'_t)(F_t)+2\nabla\log \phi_1(Y'_t)(F_t)+ 2tn_k\left(\xi_t\right)\nabla_{ e_n \oplus 0}F_t-2tn_k\left(\xi_t\right)\nabla_{ 0 \oplus e_n}F_t  \right\}dt.
				\end{align*} 
    
			We begin with first derivative terms. Use the variation from Section 2.6 for the vector $\nabla \omega(X_t') \oplus 0$, 
   denote $\eta(r,s)$ the variation, and  $T$ the unit variation field,  see \eqref{T}. 
				
				
\begin{eqnarray*}
	\lefteqn{\nabla\log \phi_1(X'_t)(F_t)} \\
	& = &  \langle  \nabla\omega(Y'_t), \nabla_{\tfrac{\partial}{\partial r}} T(r, s)  \rangle|_{r=0, s= \xi_t}  - \left(\langle \nabla_{\nabla \omega} \nabla\omega(r,s), T(r,s)  \rangle +\langle  \nabla\omega(X'_t), \nabla_{\tfrac{\partial}{\partial r}} T(r,s)  \rangle \right)|_{r=0, s= -\xi_t} \\
\end{eqnarray*}
  By \cite[Page 363]{MR3960269} $$\nabla_{\tfrac{\partial}{\partial r}} T(r, s)|_{r=0} = -\langle \gamma'(s), J'(s) \rangle e_n + J'(s),$$ where $J(s)$ is the Jacobi field along $\gamma$ with $J(-\xi_t) = \nabla \omega (X_t'), \ J(\xi_t) =0$. 
  For $\mathbb M_k^n$, $$J(s) = \left( -\frac{\langle \nabla\omega(X'_t), e_n\rangle}{2\xi_t} s + \frac{\langle \nabla\omega(X'_t), e_n\rangle}{2} \right) e_n   + \sum_{i=1}^{n-1} \langle \nabla\omega(X'_t), e_i \rangle \frac{sn_k(\xi_t -s)}{sn_k(2\xi_t)} e_i. $$
  Hence \[
  J'(s) = -\frac{\langle \nabla\omega(X'_t), e_n\rangle}{2\xi_t}  e_n   - \sum_{i=1}^{n-1} \langle \nabla\omega(X'_t), e_i \rangle \frac{cn_k(\xi_t -s)}{sn_k(2\xi_t)} e_i,
  \]
  $$\nabla_{\tfrac{\partial}{\partial r}} T(r, s)|_{r=0} = - \sum_{i=1}^{n-1} \langle \nabla\omega(X'_t), e_i \rangle \frac{cn_k(\xi_t -s)}{sn_k(2\xi_t)} e_i.$$
Plug these in, we have 
\begin{eqnarray*}
	\lefteqn{\nabla\log \phi_1(Y'_t)(F_t)} \\
	& = &  \sum_{i=1}^{n-1}\left( \langle \nabla\omega(Y'_t), e_i \rangle \langle  \nabla\omega(X_t' ), e_i \rangle \frac{-1}{sn_k(2\xi_t)}    + \langle \nabla\omega(Y'_t), e_i \rangle^2  \frac{cn_k(2\xi_t)}{sn_k(2\xi_t)}  \right) \\
	& & +\langle \nabla_{\nabla \omega} \nabla\omega(r,s), \gamma'(Y_t')  \rangle|_{r=0, s=\xi_t}.
\end{eqnarray*}
                    Similarly 
\begin{eqnarray*}
	\lefteqn{\nabla\log \phi_1(Y'_t)(F_t)} \\
	& = &  \sum_{i=1}^{n-1}\left( \langle \nabla\omega(Y'_t), e_i \rangle \langle  \nabla\omega(X_t' ), e_i \rangle \frac{-1}{\mathrm{sn}_k(2\xi_t)}    + \langle \nabla\omega(Y'_t), e_i \rangle^2  \frac{\mathrm{cn}_k(2\xi_t)}{\mathrm{sn}_k(2\xi_t)}  \right) \\
	& & +\langle \nabla_{\nabla \omega} \nabla\omega(r,s), \gamma'(Y_t')  \rangle|_{r=0, s=\xi_t}.
\end{eqnarray*}
  
                   Together with the fact that $\nabla_{\nabla \omega} \nabla\omega=\frac{1}{2}\nabla\Vert \nabla\omega \Vert^2$, we get
				\begin{align}
					&2\nabla\log \phi_1(X'_t)(F_t)+2\nabla\log \phi_1(Y'_t)(F_t)\nonumber\\
					&=\langle \nabla \Vert\nabla \omega(Y'_t) \Vert^2 , \gamma'(Y'_t) \rangle-\langle \nabla \Vert\nabla \omega(X'_t) \Vert^2 , \gamma'(X'_t) \rangle\nonumber\\
					&+2\sum_{i=1}^{n-1}\langle \nabla\omega(Y'_t), e_i \rangle\left(\frac{cs_k(2\xi_t)}{sn_k(2\xi_t)}\langle  \nabla\omega(Y'_t), e_i  \rangle -\frac{1}{sn_k(2\xi_t)}\langle \nabla\omega(X'_s), e_i  \rangle  \right)\nonumber\\
					&-2\sum_{i=1}^{n-1}\langle \nabla\omega(X'_t), e_i \rangle\left(\frac{1}{sn_k(2\xi_t)}\langle  \nabla\omega(Y'_t), e_i  \rangle -\frac{cs_k(2\xi_t)}{sn_k(2\xi_t)}\langle \nabla\omega(X'_t), e_i  \rangle  \right).\label{1st order part 1}
				\end{align}
				For the other first order term, we have
				\begin{align}
					&2tn_k\left(\xi_t\right)\nabla_{ e_n \oplus 0}F_t-2tn_k\left(\xi_t\right)\nabla_{ 0 \oplus e_n}F_t\nonumber\\
					&=-2tn_k\left(\xi_t\right)\langle \nabla_{e_n}\nabla \omega(Y'_t), e_n \rangle-2tn_k\left(\xi_t\right)\langle \nabla_{e_n}\nabla \omega(X'_t), e_n \rangle.\label{1st order part 2}
				\end{align}
				For the second order derivative, we argue in the exact same way as equation (3.18) of \cite{MR3960269} and get
				\begin{align}
					&\sum_{i=1}^{n}\nabla_{E_i}\nabla_{E_i}F_t(X'_t, Y'_t)=\langle \Delta\nabla\omega(Y'_t), e_n \rangle-\langle \Delta\nabla\omega(X'_t), e_n \rangle \nonumber\\
					-&2tn_k(\xi_t)\sum_{i=1}^{n-1}\left(\langle \nabla_{e_i}\nabla\omega(Y'_s), e_i \rangle+\langle \nabla_{e_i}\nabla\omega(X'_t), e_i \rangle   \right)\nonumber\\
					-&(n-1)tn^2_k(\xi_t)\left(\langle \nabla\omega(Y'_t), e_n \rangle-\langle \nabla\omega(X'_t), e_n \rangle   \right)\nonumber\\
					=&\langle \Delta\nabla\omega(Y'_t), e_n \rangle-\langle \Delta\nabla\omega(X'_t), e_n \rangle \nonumber\\
					-&2tn_k(\xi_t)\left(\Delta\omega(Y'_t)+\Delta\omega(X'_t) \right)+2tn_k\left(\xi_t\right)\langle \nabla_{e_n}\nabla \omega(Y'_t), e_n \rangle+2tn_k\left(\xi_t\right)\langle \nabla_{e_n}\nabla \omega(X'_t), e_n \rangle\nonumber\\
					-&(n-1)tn^2_k(\xi_t)\left(\langle \nabla\omega(Y'_t), e_n \rangle-\langle \nabla\omega(X'_t), e_n \rangle   \right).\nonumber\\
					=&\langle \Delta\nabla\omega(Y'_t), e_n \rangle-\langle \Delta\nabla\omega(X'_t), e_n \rangle \nonumber\\
					+&2tn_k(\xi_t)\left(2\lambda_1+\Vert \nabla\omega(Y'_t) \Vert^2+\Vert \nabla\omega(X'_t) \Vert^2-V(X_t')-V(Y_t')\right)\nonumber\\
					+&2tn_k\left(\xi_t\right)\langle \nabla_{e_n}\nabla \omega(Y'_t), e_n \rangle+2tn_k\left(\xi_t\right)\langle \nabla_{e_n}\nabla \omega(X'_t), e_n \rangle\nonumber\\
					-&(n-1)tn^2_k(\xi_t)\left(\langle \nabla\omega(Y'_t), e_n \rangle-\langle \nabla\omega(X'_t), e_n \rangle   \right).\label{1st order part 3}
				\end{align} 
				Here, we used $\Delta\omega=-\lambda_1-\Vert\nabla\omega\Vert^2+V$ in the last equation. Finally, note that $\nabla\Vert \nabla\omega \Vert^2=-\nabla\Delta \omega+\nabla V$. Then, equation \eqref{dF for log-concavity} follows from adding up \eqref{1st order part 1} \eqref{1st order part 2} \eqref{1st order part 3}, the Bochner–Weitzenb\"ock formula 
				\[ \Delta\nabla\omega-\nabla\Delta\omega=Ric(\nabla\omega,\cdot)=(n-1)k\langle \nabla\omega, \cdot\rangle,  \]
				and the identity $tn_k(\xi_t)=\frac{1-cs_k(2\xi_t)}{sn_k(2\xi_t)}$.
			\end{proof}
			Now we are ready to prove Theorem~\ref{log-concavity}. 
			\begin{proof}[Proof of Theorem~\ref{log-concavity}]
            
				First, we pick $0<D<D'<\pi/\sqrt{k}$ and continuously extend $\tilde{V}$ to an even function on $[-D'/2,D'/2]$. Let $\bar{\phi}_1$ be the first eigenfunction of the 1-dimensional model on $[-D'/2,D'/2]$. Moreover, let $\{Q_i\}_{1\leq i\leq n} $ be the Jacobi fields of the geodesic variations $\eta(r,s)$ associated with the vector fields $\{E_i\}_{1\leq i\leq n}$, which we introduced at the beginning of this section. For simplicity, we denote $\Psi=(\log\bar{\phi}_1)'$. 
				
				By It\^o's formula \eqref{Ito's formula} 
            and the second variation formula, 
				\begin{align}
					d\Psi(\xi_t)&=d\{\text{martingale}\}\nonumber\\
					&+\{\frac{1}{2}\sum_{i=1}^{n}I(Q_i,Q_i)\Psi'+F_t\Psi'+\Psi'' -2tn_k\left(\xi_t\right)\Psi'\}dt.\label{dPsi for log-concavity}
				\end{align}
				For $\mathbb M^n_k$, we have
				\begin{equation}
					\label{eq2 in the proof}
					\frac{1}{2}\sum_{i=1}^{n}I(Q_i,Q_i) = -(n-1)tn_k(\xi_t).
				\end{equation}
				On the other hand, a direct computation (similar to Lemma 2.7 of \cite{MR3960269}) gives that
				\begin{equation}
					\label{ODE for Psi}
					\Psi''+2\Psi'\Psi-tn_k\left((n+1)\Psi'+2\bar{\lambda}_1+2\Psi^2-2\Tilde{V} \right)-\Tilde{V}'-(n-1)(K-tn_k^2)\Psi=0.
				\end{equation}
				Combining estimates \eqref{ODE for Psi}\eqref{eq2 in the proof} with \eqref{dF for log-concavity} \eqref{dPsi for log-concavity} gives 
				\begin{align}
					d\left(F_t-2\Psi(\xi_t)\right)&= d\{\text{martingale}\} + 4 tn_k\,  (\lambda_1 - \bar{\lambda}_1) \nonumber\\
					& +\left( (n-1)(K-tn_k^2)-2\Psi'\right)(F_t-2\Psi)-4tn_k(\xi_t)\Psi^2 dt\nonumber\\
					& +\left\{ \langle \nabla V(Y'_t), e_n\rangle-\langle \nabla V(X'_t), e_n\rangle-2\tilde{V}'(\xi_t)\right\}+2 tn_k(\xi_t)(2\tilde{V}(\xi_t)-V(X'_t)-V(Y'_t))dt\nonumber\\
					&+2tn_k(\xi_t)\left(\langle  \nabla\omega(Y'_t), e_n  \rangle^2+\langle \nabla\omega(X'_t), e_n  \rangle^2 \right)dt\nonumber\\
					&+\frac{2}{sn_k(2\xi_t)}\sum_{i=1}^{n-1}\left(\langle  \nabla\omega(Y'_t), e_i  \rangle-\langle \nabla\omega(X'_t), e_i  \rangle \right)^2dt\nonumber\\
					&\geq d\{\text{martingale}\}+ \left( (n-1)(K-tn_k^2(\xi_t))-2\Psi'\right)(F_t-2\Psi)dt\nonumber\\
					&+tn_k(\xi_t)(F_t+2\Psi)(F_t-2\Psi)dt\nonumber\\
					&=d\{\text{martingale}\}\nonumber\\
					&+\left( (n-1)(K-tn_k^2(\xi_t))-2\Psi'+tn_k(\xi_t)(F_t+2\Psi)\right)(F_t-2\Psi)dt. \label{log-concavity inequality differential form}
				\end{align}
			Here for the inequality we used the assumptions that $\tilde{V}$ is a modulus of convexity of $V$ \eqref{eq:mod-of-convexity}, 
   \eqref{eqn:condition om lambda and V}, $k \ge 0$ (this is the only place $k\ge 0$ is used),
    and \[
    \langle  \nabla\omega(Y'_t), e_n  \rangle^2+\langle \nabla\omega(X'_t), e_n  \rangle^2 \ge \frac{\left(\langle  \nabla\omega(Y'_t), e_n  \rangle - \langle \nabla\omega(X'_t), e_n  \rangle \right)^2 }{2} = \frac{F_t^2}{2}. 
                \]  
                Next, we observe that \eqref{log-concavity inequality differential form} is equivalent to 
				\begin{align*}
					&d\left( e^{\int_{0}^{t}2\Psi'-(n-1)(K-tn_k^2(\xi_s))-tn_k(\xi_s)(F_s+2\Psi)ds}(F_t-2\Psi(\xi_t))  \right)\\
					&\geq e^{\int_{0}^{t}2\Psi'-(n-1)(K-tn_k^2(\xi_s))-tn_k(\xi_s)(F_s+2\Psi)ds}d\{\text{martingale}\}.
				\end{align*}
                 Let $N>0$, integrating from $0$ to $\tau'\wedge N:= \min\{\tau', N \}$ gives 
                \begin{equation*}
					\left( e^{\int_{0}^{\tau'\wedge N}2\Psi'-(n-1)(K-tn_k^2(\xi_s))-tn_k(\xi_s)(F_s+2\Psi)ds}(F_{\tau'\wedge N}-2\Psi(\xi_{\tau'\wedge N})  \right)\geq \left(F_0-2\Psi(\xi_0)\right)+\text{Martingale}_{\tau'\wedge N}.
				\end{equation*}
                Since $\tau' \wedge N$ is bounded, the stopped martingale $\text{Martingale}_{\tau'\wedge N}$ in the above inequality is another martingale. We can thus take expectation and get
                \begin{align*}
					F_0-2\Psi(\xi_0)&\leq \mathbb{E}\left( e^{\int_{0}^{\tau'\wedge N}2\Psi'-(n-1)(K-tn_k^2(\xi_s))-tn_k(\xi_s)(F_s+2\Psi)ds}(F_{\tau'\wedge N}-2\Psi(\xi_{\tau'\wedge N})  \right)\\
                                                &\leq \mathbb{E}\left( e^{\int_{0}^{\tau'\wedge N}2\Psi'-(n-1)(K-tn_k^2(\xi_s))-tn_k(\xi_s)(F_s+2\Psi)ds}\abs{F_{\tau'\wedge N}-2\Psi(\xi_{\tau'\wedge N})} \right).
				\end{align*}
                By Fatou,
                \begin{equation*}
                    F_0-2\Psi(\xi_0)\leq \mathbb{E}\left( \liminf_{N\rightarrow \infty} e^{\int_{0}^{\tau'\wedge N}2\Psi'-(n-1)(K-tn_k^2(\xi_s))-tn_k(\xi_s)(F_s+2\Psi)ds}\abs{F_{\tau'\wedge N}-2\Psi(\xi_{\tau'\wedge N})} \right).
                \end{equation*}
                Since $\tau'<\infty$ almost surely by Lemma \ref{lem:successful coupling for modified diffusion}, we have $\tau' \wedge N \rightarrow \tau'$ as $N$ approaches infinity. On the other hand, $F_{\tau'}=\Psi(\xi_{\tau'})=0$. Thus,
                \begin{align*}
                    &\mathbb{E}\left( \liminf_{N\rightarrow \infty} e^{\int_{0}^{\tau'\wedge N}2\Psi'-(n-1)(K-tn_k^2(\xi_s))-tn_k(\xi_s)(F_s+2\Psi)ds}\abs{F_{\tau'\wedge N}-2\Psi(\xi_{\tau'\wedge N})} \right)\\
                    &=\mathbb{E}\left(  e^{\int_{0}^{\tau'}2\Psi'-(n-1)(K-tn_k^2(\xi_s))-tn_k(\xi_s)(F_s+2\Psi)ds}\abs{F_{\tau'}-2\Psi(\xi_{\tau'})} \right)=0.
                \end{align*}
                We conclude that
				\begin{equation*}
					F_0-2\Psi(\xi_0)=\langle \nabla\log \phi_1(y), \gamma'(y)  \rangle -\langle \nabla\log \phi_1(x), \gamma'(x)  \rangle-2\Psi(\rho(x,y)/2)\leq 0.  
				\end{equation*}
				Finally, we send $D'\rightarrow D$. The proof is complete.
			\end{proof}
                \begin{rema}
                    Our proof should be compared with that of Theorem 3.2 of \cite{MR3489850}. In particular, we achieved two important improvements besides the handling of all the curvature terms. First, we did not assume the potential $V$ is convex. Our proof with optional stopping technique does not require $\Psi'$ to be non-positive (see remark 3.3 of \cite{MR3489850}). Second, our argument does not require the first eigenfunction $\phi_1$ to be log-concave, which was crucial in the proof in \cite{MR3489850}.
                \end{rema}
                \begin{rema}
                    It is worth mentioning that at the critical point where $F_t=2\Psi(\xi_t)$, the quantity $2\Psi'-(n-1)(K-tn_k^2(\xi_t))-tn_k(\xi_t)(F_t+2\Psi)$, which appears in \eqref{log-concavity inequality differential form} and later as the integrand in the exponential function becomes $2\Psi'-4tn_k(\xi_t)\Psi-(n-1)(K-tn_k^2(\xi_t))$. The latter appeared in Theorem 3.8 of \cite{MR3960269} (see also Remark 3.9 in the same reference) and was used in the proof of Theorem 1.6 of \cite{MR3960269}. The diameter restriction $D<\pi/(2\sqrt{k})$ was coming from the need to make this quantity non-positive. In our proof above, this sign restriction, and hence the diameter restriction, was removed thanks to the optional stopping technique. 
                \end{rema}

			\section{Fundamental gap comparison}
			In this section, we prove Theorem \ref{thm:comparison}. Let $\phi_1,\ \phi_2 $ be the eigenfunctions associated with the first two eigenvalues of  the Schr\"odinger operator $-\Delta+V$ on $\Omega$ respectively. Recall that the ground state transform
			\begin{equation}
				\label{Ground state transform}
				v_t=\frac{e^{-\lambda_2 t}\phi_2}{e^{-\lambda_1 t}\phi_1}=e^{-(\lambda_2-\lambda_1)t}\frac{\phi_2}{\phi_1}
			\end{equation}
			is a smooth solution to 
			\begin{align}
				\partial_tv&=Lv=\Delta v+2\langle \nabla\log\phi_1, \nabla v \rangle\label{Ground state equation} \\
				v(0,\cdot)&=\phi_2/\phi_1\nonumber
			\end{align}
			on $\mathbb{R}^+\times \bar{\Omega}$. 
   We will use the coupling 
   $X_t, Y_t,$ from \eqref{couping1}-\eqref{coupling2}.
			\begin{proof}[Proof of Theorem \ref{thm:comparison}]
				 Let $\rho_t=\rho(X_t,Y_t)$ and $\xi_t = \rho_t/2$. By It\^o's formula,
				\begin{equation}
					\label{SDE for rho}
					d\rho_t=2\sqrt{2}d\beta_t+\left(\sum_{i=1}^{n}I(Q_i,Q_i)+2\langle \nabla\log\phi_1(Y_t), \gamma'(Y_t)  \rangle -2\langle \nabla\log\phi_1(X_t), \gamma'(X_t)  \rangle \right)dt,
				\end{equation}
				where $\beta_t$ is a Brownian motion independent from $B_t$. Let $\Tilde{\phi}_1, \Tilde{\phi}_2$ be the first two eigenfunctions of the corresponding one-dimensional model. Define $\Phi=\Tilde{\phi}_2/\Tilde{\phi}_1$ and let
                \[   F_t=\langle \nabla\log\phi_1(Y_t), \gamma'(Y_t)  \rangle -\langle \nabla\log\phi_1(X_t), \gamma'(X_t)  \rangle. \]
                With \eqref{SDE for rho} at hand, we apply It\^o's formula again and get 
				\begin{align*}
					d\Phi(\xi_t)=\sqrt{2}\Phi'(\xi_t)d\beta_t+\left(\frac{1}{2}\Phi'(\xi_t)\sum_{i=1}^{n}I(Q_i,Q_i)+\Phi'(\xi_t)F_t+\Phi''(\xi_t) \right)dt.
				\end{align*}
				By our assumption, $\log\Tilde{\phi}_1$ is a modulus of concavity of $\log\phi_1$. Thus, we have $F_t\leq 2(\log\tilde{\phi}_1)'(\xi_t)$. Moreover, by Index lemma
				\[ \frac{1}{2}\sum_{i=1}^{n}I(Q_i,Q_i)\leq -(n-1)tn_k(\xi_t).  \]
				As a result,
				\begin{align}
					\label{Phi without index form}
					d\Phi(\xi_t)\leq\sqrt{2}\Phi'(\xi_t)d\beta_t+\left(-\Phi'(\xi_t)(n-1)tn_k(\xi_t)+\Phi'(\xi_t)2(\log\tilde{\phi}_1)'(\xi_t)+\Phi''(\xi_t) \right)dt.
				\end{align}
				It is easily checked that equation (2.7) of \cite{MR3960269} still holds for the Schr\"odinger operator $-\Delta+V$. Hence, we have
				\begin{equation}
					\label{ODE for Phi}
					\Phi''(s)-\Phi'(s)(n-1)tn_k(s)+\Phi'(s)2(\log\tilde{\phi}_1)'(s)=-(\bar{\lambda}_2-\bar{\lambda}_1)\Phi(s).
				\end{equation} 
				Combining \eqref{Phi without index form} and \eqref{ODE for Phi} gives
				\begin{equation*}
					d\Phi(\xi_t)\leq\sqrt{2}\Phi'(\xi_t)d\beta_t-(\bar{\lambda}_2-\bar{\lambda}_1)\Phi(\xi_t)dt,
				\end{equation*}
				which is equivalent to 
				\begin{equation*}
					d\left(e^{(\bar{\lambda}_2-\bar{\lambda}_1)t}\Phi(\xi_t) \right)\leq\sqrt{2}e^{(\bar{\lambda}_2-\bar{\lambda}_1)t}\Phi'(\xi_t)d\beta_t.
				\end{equation*}
				Integrating and taking expectation gives
				\begin{equation*}
					\mathbb{E}\left(e^{(\bar{\lambda}_2-\bar{\lambda}_1)t}\Phi(\xi_t) \right)\leq \Phi(\xi_0),
				\end{equation*}
				or equivalently
				\begin{equation}
					\label{Bound on Phi(xi_t)}
					\mathbb{E}\Phi(\xi_t)\leq e^{-(\bar{\lambda}_2-\bar{\lambda}_1)t}\Phi(\xi_0).
				\end{equation}
                By Lemma \ref{lem:process}, we know that the above inequality holds for all $t\geq 0$. Finally, by the definition \eqref{Ground state transform} and It\^o's formula, we have $ v_t(x)=\mathbb{E}(v_0(X_t))$. Since $v_0$ is Lipschitz on $\bar{\Omega}$, we can find $K>0$ such that
				\begin{align}
					\abs{v_t(x)-v_t(y)}&=\abs{\mathbb{E}(v_0(X_t)-v_0(Y_t))}\leq\mathbb{E}\abs{v_0(X_t)-v_0(Y_t)}\nonumber\\
					&\leq K \mathbb{E}\rho(X_t,Y_t)=2K\mathbb{E}\xi_t. \label{Lipschitz estimate}
				\end{align}
				Since $\Tilde{V}$ is even and \eqref{ODE for Phi} is the same as equation (2.7) of \cite{MR3960269}, Lemma 2.3 of \cite{MR3960269} applies to $\Phi$. Thus, $\Phi$ is increasing on $[0,D/2]$ with $\Phi'(0)>0$; from which we deduce that there exists $c_1>0$, such that $\Phi(s)\geq c_1 s$ on $[0,D/2]$. Hence, we have from \eqref{Bound on Phi(xi_t)}, that
				\[c_1\mathbb{E}\xi_t\leq  \mathbb{E}(\Phi(\xi_t))\leq e^{-(\tilde{\lambda}_2-\tilde{\lambda}_1)t}\Phi(\xi_0).  \] 
				Putting it together with \eqref{Lipschitz estimate} and the definition of $v_t$ gives
				\[ e^{-(\lambda_2-\lambda_1)t}\abs{v_0(x)-v_0(y)}=\abs{v_t(x)-v_t(y)}\leq \frac{2K}{c_1}e^{-(\bar{\lambda}_2-\bar{\lambda}_1)t}\Phi(\xi_0).    \]
				Since $v_0$ is not constant function, we can find $x\neq y$ such that $\abs{v_0(x)-v_0(y)}>0$. Our result follows by sending $t\rightarrow \infty$.
			\end{proof}
   \begin{rema}
This proof is in the flavor of parabolic proof of Theorem 4.1 in \cite{MR3960269}, but simplifies as we do not need to  use the preservation of  modulus of continuity, Theorem 4.6 in  \cite{MR3960269}.  
   \end{rema}

			\bibliographystyle{plain}
			\bibliography{reference}
			
		\end{document}